\documentclass[oneside,10pt]{article}          
\usepackage[b5paper]{geometry}	       
\usepackage{amsfonts,amsmath,latexsym,amssymb} 
\usepackage{enumerate}
\usepackage{amsthm}                
\usepackage{mathrsfs,upref}         
\usepackage{mathptmx}		    
\usepackage{oam}	           

\newtheorem{theorem}{Theorem}[section]     
\newtheorem{lemma}[theorem]{Lemma}               

\theoremstyle{definition}
\newtheorem{definition}[theorem]{Definition}
\newtheorem*{remarks}{Remarks}
\newtheorem*{acknowledgement}{Acknowledgement}

\newcommand{\CC}{{\mathbb C}}
\newcommand{\DD}{{\mathbb D}}
\newcommand{\RR}{{\mathbb R}}
\newcommand{\TT}{{\mathbb T}}
\newcommand{\cB}{{\cal B}}
\let\Re\undefined
\DeclareMathOperator{\Re}{Re}

\begin{document}

\title{Spectral sets for numerical range}

\author{Hubert Klaja, Javad Mashreghi and Thomas Ransford}

\address{Hubert Klaja, D\'epartement de math\'ematiques et de statistique, Universit\'e Laval, 
1045 avenue de la M\'edecine, Qu\'ebec (QC), Canada G1V 0A6\\
\email{hubert.klaja@gmail.com}}

\address{Javad Mashreghi, D\'epartement de math\'ematiques et de statistique, Universit\'e Laval, 
1045 avenue de la M\'edecine, Qu\'ebec (QC), Canada G1V 0A6\\
\email{javad.mashreghi@mat.ulaval.ca}}

\address{Thomas Ransford, D\'epartement de math\'ematiques et de statistique, Universit\'e Laval, 
1045 avenue de la M\'edecine, Qu\'ebec (QC), Canada G1V 0A6\\
\email{ransford@mat.ulaval.ca}}

\CorrespondingAuthor{Thomas Ransford}

\date{09.08.2016}                             

\keywords{numerical range; spectral set; dilation;}

\subjclass{47A12, 47A25}

\thanks{Research of JM supported by a grant from NSERC. Research of TR supported by grants from NSERC and the Canada Research Chairs program} 
       
\begin{abstract}
We define and study a numerical-range analogue of the notion of spectral set. 
Among the results obtained are a positivity criterion and a dilation theorem,
analogous to those already known for spectral sets. An important difference from the classical definition
is the role played in the new definition by the base point. We present some examples 
to illustrate this aspect.
\end{abstract}

\maketitle

\section{Introduction}

Let $H$ be a complex Hilbert space. 
We denote by $B(H)$ the algebra of bounded linear operators on $H$.
Also, we write $\DD$ for the open unit disk.

\subsection{Spectral sets}

Our starting point is the following well-known inequality of von Neumann \cite{vN51}.

\begin{theorem}\label{T:vN}
Let $T\in B(H)$ with $\|T\|\le1$. 
Then, for every rational function $f$ with poles outside $\overline{\DD}$, we have 
\begin{equation}\label{E:vN}
\|f(T)\|\le\sup_{z\in\overline{\DD}}|f(z)|.
\end{equation}
\end{theorem}

The following definition, also due to von Neumann \cite{vN51}, 
makes abstraction of this property.
We write $\sigma(T)$ for the spectrum of $T$.

\begin{definition}\label{D:ss}
Let $T\in B(H)$. A subset $S$ of $\CC$ is a {\em spectral set} for $T$ 
if $\sigma(T)\subset S$
and if, for every rational function $f$ with poles outside $\overline{S}$, we have
\[
\|f(T)\|\le \sup_{z\in S}|f(z)|.
\]
\end{definition}

Thus Theorem~\ref{T:vN} says that
$\overline{\DD}$ is a spectral set for all $T$ with $\|T\|\le 1$.

Spectral sets have been studied extensively over the years, 
especially in relation to dilation theorems.
The following theorem summarizes some known equivalences.
Given $T\in B(H)$, we write $\Re T:=(T+T^*)/2$, 
where $T^*$ denotes the Hilbert-space adjoint of $T$. 
Also, if $K$ is Hilbert space containing $H$ (as a closed subspace),
then we write $P_H:K\to H$ for the orthogonal projection of $K$ onto $H$. 

\begin{theorem}\label{T:ss}
Let $\Omega$ be a Jordan domain,
and let $T\in B(H)$ with $\sigma(T)\subset \Omega$. 
The following statements are equivalent:
\begin{enumerate}[\rm(i)]
\item $\overline{\Omega}$ is a spectral set for $T$.
\item There exists $z_0\in\Omega$ such that $\Re (h(T))\ge0$ for all 
conformal maps $h$ of $\Omega$ onto the right half-plane with $h(z_0)=1$.
\item
There exist a Hilbert space $K$ containing $H$ 
and a normal operator $N\in B(K)$ such that
$\sigma(N)\subset\partial\Omega$ and such that, 
for every rational function $f$ with poles outside $\overline{\Omega}$, 
we have $f(T)=P_Hf(N)|_H$.
\end{enumerate}
\end{theorem}

Property (iii) is often expressed by saying that $N$ is a {\em normal dilation} of $T$.
 It clearly implies property (i) since, if (iii) holds, then
\[
\|f(T)\|
=\|P_Hf(N)|_H\|\le \|f(N)\|
=\sup_{z\in \sigma(N)}|f(z)|\le\sup_{z\in\overline{\Omega}}|f(z)|.
\]

The equivalence between properties (i) and (iii) is due independently 
to Berger \cite{Be63}, Foia\c s \cite{Fo59} and Lebow \cite{Le63}. 
The equivalence with the (formally weaker) property (ii) 
was studied by Agler, Harland and Raphael in \cite{AHR08}, 
where one can also find an extensive discussion of 
generalizations of this result to multiply connected domains.

As an example, let us consider the case $\Omega=\DD$. 
A conformal map $h$ of $\DD$ onto the right half-plane such that $h(0)=1$ has the form
$h(z)=(e^{i\theta}+ z)/(e^{i\theta} -z)$.
Thus condition (ii) (with $z_0=0$) becomes 
$\Re((I+e^{-i\theta} T)(I-e^{-i\theta} T)^{-1})\ge0$ for all $e^{i\theta}\in\TT$. 
Now, for each $\theta$, we have
\begin{align*}
\Re&(( I+e^{-i\theta} T)(I-e^{-i\theta} T)^{-1})\ge0\\
&\iff \Re\langle(I+e^{-i\theta} T)(I-e^{-i\theta} T)^{-1}x,x\rangle\ge0 \quad \text{for all~}x\in H\\
&\iff \Re\langle(I+e^{-i\theta} T)y,(I-e^{-i\theta} T)y\rangle\ge0  \quad \text{for all~}y\in H\\
&\iff \|y\|^2-\|Ty\|^2\ge0\quad \text{for all~}y\in H\\
&\iff \|T\|\le1.
\end{align*}
Thus, thanks to the equivalence between (i) and (ii), 
we recover the fact that $\overline{\DD}$ is a spectral set for $T$ if $\|T\|\le 1$.
In addition, the equivalence
between (i) and (iii) shows that, if $\|T\|\le 1$, 
then $T$ has a normal dilation $N$ with $\sigma(N)\subset\TT$,
and so we recover the celebrated theorem of Sz-Nagy \cite{SN53} 
that every contraction has a unitary dilation.
(Our argument supposes that  $\sigma(T)\subset\DD$, 
but this restriction can be removed by considering $rT$ for $r<1$ and then letting $r\to1^-$). 

\subsection{$W$-spectral pairs}

Our purpose in this article is to carry out a similar program for numerical ranges. 
We recall that, given $T\in B(H)$, the {\em numerical range} and {\em numerical radius} 
of $T$ are defined respectively by
\[
W(T):=\{\langle Tx,x\rangle: x\in H,~\|x\|=1\}
\quad\text{and}\quad
w(T):=\sup_{z\in W(T)}|z|.
\]
It is well known that $W(T)$ is a convex set satisfying $\sigma(T)\subset\overline{W(T)}$. 
Also, we always have $\|T\|/2\le w(T)\le\|T\|$. 
For general background on numerical ranges, we refer to
the book of Gustafson and Rao \cite{GR97}.

This time our point of departure is the following result 
due to Berger and Stampfli \cite{BS67} and Kato \cite{Ka65}.

\begin{theorem}\label{T:BSK}
Let $T\in B(H)$ with $w(T)\le1$. 
Then, for each rational function $f$ with poles outside $\overline{\DD}$ and satisfying $f(0)=0$, 
we have
\[
w(f(T))\le \sup_{z\in\overline{\DD}}|f(z)|.
\]
\end{theorem}

As remarked in \cite{BS67}, this result is no longer true if one omits  the condition $f(0)=0$. 
We shall return to this point in \S\ref{S:basept} below.
Motivated by this result, we make the following definition.

\begin{definition}\label{D:Wsp}
Let $T\in B(H)$, let $S$ be a subset of $\CC$ and let $z_0\in S$. 
We say that $(S,z_0)$ is a {\em $W$-spectral pair} for $T$ if  $\sigma(T)\subset S$
and if, for every rational function $f$ with poles outside $\overline{S}$ and satisfying $f(z_0)=0$, 
we have
\[
w(f(T))\le \sup_{z\in S}|f(z)|.
\]
\end{definition}

Thus Theorem~\ref{T:BSK} just says that 
$(\overline{\DD},0)$ is a $W$-spectral pair for $T$ if $w(T)\le1$.

Our main result is the following analogue of Theorem~\ref{T:ss}.

\begin{theorem}\label{T:Wsp}
Let $\Omega$ be a Jordan domain, let $z_0\in\Omega$,
and let $T\in B(H)$ with $\sigma(T)\subset \Omega$. 
The following  statements are equivalent:
\begin{enumerate}[\rm(i)]
\item $(\overline{\Omega},z_0)$ is a $W$-spectral pair for $T$.
\item $\Re (h(T))\ge -I$ for all 
conformal maps $h$ of $\Omega$ onto the right half-plane with $h(z_0)=1$.
\item There exist a Hilbert space $K$ containing $H$ and a normal operator $N\in B(K)$ such that
$\sigma(N)\subset\partial\Omega$ and such that, 
for every rational function $f$ with poles outside $\overline{\Omega}$
and satisfying $f(z_0)=0$, 
we have $f(T)=2P_Hf(N)|_H$.
\end{enumerate}
\end{theorem}

Once again, let us consider what this theorem tells us 
in the case when $\Omega=\DD$ and $z_0=0$.
Condition~(ii) now becomes 
$\Re((I+e^{-i\theta} T)(I-e^{-i\theta} T)^{-1})\ge -I$ for all $e^{i\theta}\in\TT$.
A computation similar to the one performed earlier shows that 
this is equivalent to $w(T)\le1$.
Thus, thanks to the equivalence between (i) and (ii), 
we recover the Berger--Stampfli--Kato theorem that 
$(\overline{\DD},0)$ is a $W$-spectral pair for $T$ if $w(T)\le 1$.
Also the equivalence with (iii) shows that, if $w(T)\le1$, 
then $T$ has a $2$-unitary dilation in the sense that there exists a
unitary operator $U$ on a Hilbert space $K$ containing $H$ such that  
$T^n=2P_HU^n|_H$ for all $n\ge1$.
This is the so-called `strange dilation theorem' of Berger \cite{Be63,BS67}.

The proof of Theorem~\ref{T:Wsp} will be presented in \S\ref{S:pfWsp}.

\subsection{The role of the base point}\label{S:basept}

As remarked earlier, 
Theorem~\ref{T:BSK} only works if $f(0)=0$.
This explains the presence of the base point $z_0$ in Definition~\ref{D:Wsp}.
We now make some further comments on this phenomenon 
and discuss some examples.

The sharp version of Theorem~\ref{T:BSK} in the case $f(0)\ne0$ 
was found by Drury \cite{Dr08}.
He proved that, if $W(T)\subset\overline{\DD}$ 
and $f$ is a rational function mapping $\overline{\DD}$ into $\overline{\DD}$,
then the numerical range of $f(T)$ is contained in $t(f(0))$,
where $t(a)$ is the `teardrop' set formed by 
taking the  convex hull of the union of closed disks
$\overline{D}(0,1)\cup \overline{D}(a,1-|a|^2))$.
For another proof, see \cite{KMR16}. 
The following theorem is an easy consequence of Drury's result.

\begin{theorem}
Let $T\in B(H)$ and let $(S,z_0)$ be a $W$-spectral pair for $T$.
If $f$ is a rational function such that $f(S)\subset\overline{\DD}$, 
then $W(f(T))\subset t(f(z_0))$.
\end{theorem}
 
\begin{proof}
We may assume that $|f(z_0)|<1$ 
(otherwise work with $rf$ where $r<1$, and then let $r\to1^-$ at the end).
Let $\phi$ be the  M\"obius automorphism of $\overline{\DD}$ 
that exchanges $0$ and $f(z_0)$.
Then $f=\phi\circ f_1$, 
where $f_1$ is rational, maps $\overline{\DD}$ into $\overline{\DD}$, and,
in addition, sends $0$ to $0$. 
By definition of $W$-spectral set, 
it follows that $W(f_1(T))\subset\overline{\DD}$.
By Drury's theorem applied to $\phi$, 
we deduce that $W(\phi(f_1(T)))\subset t(\phi(0))$, 
in other words, 
that $W(f(T))\subset t(f(z_0))$, as desired.
\end{proof}
 
When $S$ is not a disk, 
then  the condition that $W(T)\subset S$ 
may no longer be sufficient to guarantee that
$(S,z_0)$ be a $W$-spectral pair for $T$, 
even if $z_0$ is a `central' point of $S$. 
This is illustrated by the next theorem.

\begin{theorem}\label{T:ellipse}
Let $S:=\{x+iy:(x^2/a^2)+(y^2/b^2)\le1\}$, where $a>b>0$. 
Then there exists a $2\times 2$ matrix $T$ such that $W(T)\subset S$, 
but $(S,0)$ is not a $W$-spectral pair for $T$.
\end{theorem}

The proof of this theorem, 
which is based on an example of Michel Crouzeix,
will be presented in \S\ref{S:ellipse}.

\section{Proof of Theorem~\ref{T:Wsp}}\label{S:pfWsp}

We require two lemmas. 
The first one is a Herglotz-type representation formula.

\begin{lemma}\label{L:Herglotz}
Let $\Omega$ be a Jordan domain and let $z_0\in\Omega$.
Let $T\in B(H)$ with $\sigma(T)\subset\Omega$.
Then, for every function $f$ continuous on $\overline{\Omega}$ and 
holomorphic on $\Omega$ with $f(z_0)\in\RR$, we have
\begin{equation}\label{E:Herglotz}
f(T)=\int_{\partial\Omega}h_\zeta(T)\Re f(\zeta) \,d\omega_{z_0}(\zeta),
\end{equation}
where $\omega_{z_0}$ denotes the harmonic measure of $\Omega$ 
relative to the point $z_0$ and,
for each $\zeta\in\partial\Omega$, the function
$h_\zeta$ is the unique conformal mapping of $\Omega$ 
onto the right half-plane satisfying
$h_\zeta(z_0)=1$ and $\lim_{z\to\zeta}|h_\zeta(z)|=\infty$.
\end{lemma}

\begin{proof}
Let $g$ be a conformal mapping of $\Omega$ onto $\DD$ 
such that $g(z_0)=0$.
As $\Omega$ is a Jordan domain, 
$g$ extends to a homeomorphism of $\overline{\Omega}$ onto $\overline{\DD}$.
The function $\Re (f\circ g^{-1})$ is continuous on $\overline{\DD}$ 
and harmonic on $\DD$, 
so it satisfies Poisson's formula:
\[
\Re (f\circ g^{-1})(w)=
\int_\TT \Re\Bigl(\frac{e^{i\theta}+w}{e^{i\theta}-w}\Bigr)\Re(f\circ g^{-1})(e^{i\theta})\,\frac{d\theta}{2\pi}
\quad(w\in\DD).
\]
After the change of variables $z:=g^{-1}(w)$ and $\zeta:=g^{-1}(e^{i\theta})$,
Lebesgue measure on $\TT$ is transformed into 
harmonic measure on $\partial\Omega$ with respect to $z_0$,
and this equation becomes
\begin{align*}
\Re f(z)
&=\int_{\partial\Omega} \Re\Bigl(\frac{g(\zeta)+g(z)}{g(\zeta)-g(z)}\Bigr)\Re f(\zeta)\,d\omega_{z_0}(\zeta)\\
&=\int_{\partial\Omega}\Re h_\zeta(z)\Re f(\zeta)\,d\omega_{z_0}(\zeta)
\quad(z\in\Omega).
\end{align*}
It follows that
\[
f(z)=\int_{\partial\Omega}h_\zeta(z)\Re f(\zeta) \,d\omega_{z_0}(\zeta)\quad(z\in\Omega).
\] 
Indeed, the two sides of this last equation are holomorphic functions of $z$ 
having identical real parts,
so they differ by a constant, 
and as their imaginary parts vanish at $z=z_0$, 
the constant is zero.
Finally, applying both sides of the equation to $T$ 
via the holomorphic functional calculus,
we deduce that \eqref{E:Herglotz} holds.
\end{proof}

The second lemma is a version of Naimark's dilation theorem.
Let $X$ be a compact metric space and let $\cB(X)$ denote the Borel subsets of $X$. 
A {\em positive operator measure} on $X$ is a map $F:\cB(X)\to B(H)$ such that:
\begin{itemize}
\item $F(B)$ is a positive operator for each $B\in\cB(X)$;
\item $F(X)=I$;
\item for each pair $x,y\in H$, the map $B\mapsto \langle F(B)x,y\rangle$ is countably additive.
\end{itemize}
If further $F(B)$ is a projection for each $B\in\cB(X)$, 
then $F$ is called a {\em spectral measure}.
For general background on spectral measures, we refer to \cite{Ha57} and \cite{Pa02}.

\begin{lemma}\label{L:Naimark}
Let $X$ be a compact metric space, 
and let $F:\cB(X)\to B(H)$ be a positive operator measure.
Then there exist a Hilbert space $K$ containing $H$ 
and a spectral measure $E:\cB(X)\to B(K)$ such that
$F(B)=P_HE(B)|_H$ for all $B\in\cB(X)$.
\end{lemma}

\begin{proof}
See for example  \cite[Theorem~4.6]{Pa02}.
\end{proof}

\begin{proof}[Proof of Theorem~\ref{T:Wsp}]
First we prove the implication (ii)$\Rightarrow$(iii).
Assume that (ii) holds. 
Adopting the notation of Lemma~\ref{L:Naimark}, 
let us define $F:\cB(\partial\Omega)\to B(H)$ by
\[
F(B):=\frac{1}{2}\int_B \Re(h_\zeta(T)+I)\,d\omega_{z_0}(\zeta)
\quad(B\in\cB(\partial\Omega)).
\]
By hypothesis (ii) 
we have $\Re(h_\zeta(T)+I)\ge0$ for each $\zeta\in\partial\Omega$,
so $F$ takes positive operator values.  
Further, applying Lemma~\ref{L:Herglotz} with $f=1$,
we see that $\int_{\partial\Omega}h_\zeta(T)\,d\omega_{z_0}(\zeta)=I$,
whence also $F(X)=I$. 
Clearly, $B\mapsto \langle F(B)x,y\rangle$ is countably additive for each pair $x,y\in H$.
Therefore $F$ is a positive operator measure on $\partial\Omega$.
By Lemma~\ref{L:Naimark}, 
there exist a Hilbert space $K$ containing $H$ 
and a spectral measure $E:\cB(\partial\Omega)\to B(K)$ 
such that $F(B)=P_HE(B)|_H$ for all $B\in\cB(\partial \Omega)$.
Define $N\in B(K)$ by
\[
N:=\int_{\partial\Omega}\zeta\,dE(\zeta),
\]
where now the integral converges in the strong operator topology on $B(K)$. 
Then $N$ is a normal operator and $\sigma(N)\subset\partial\Omega$.
Further, if $f$ is a rational function with poles outside $\overline{\Omega}$ 
such that $f(z_0)=0$, then
\[
2P_Hf(N)|_H
=2P_H\Bigl(\int_{\partial\Omega}f(\zeta)\,dE(\zeta)\Bigr)|_H
=\int_{\partial\Omega}f(\zeta)\Re(h_\zeta(T)+I)\,d\omega_{z_0}(\zeta).
\]
Hence
\[
2P_H(\Re f(N))|_H=\int_{\partial\Omega}\Re f(\zeta)(\Re(h_\zeta(T)+I)\,d\omega_{z_0}(\zeta)=\Re f(T),
\]
where the last equality comes from Lemma~\ref{L:Herglotz}.
Applying the same argument to $if$ and then adding, 
we obtain finally that $2P_Hf(N)|_H=f(T)$. 
Thus property (iii) holds.

Next we prove the implication (iii)$\Rightarrow$(i).
Assume that (iii) holds. 
Let $f$ be a rational function such that $f(\overline{\Omega})\subset\overline{\DD}$ and $f(z_0)=0$.
Then, for each $\alpha\in\DD$, 
the function $ f/(1-\alpha f)$ is rational with poles outside $\overline{\Omega}$
and vanishes at $z_0$,
so by hypothesis~(iii)
\[
 f(T)(I-\alpha f(T))^{-1}
 =2P_H( f(N)(I-\alpha f(N))^{-1})|_H.
\]
Rearranging gives
\[
(I-\alpha f(T))^{-1}
=P_H((I+\alpha f(N))(I-\alpha f(N))^{-1})|_H,
\]
whence, in particular, $\Re((I-\alpha f(T))^{-1})\ge0$. 
It follows that $\Re(I-\alpha f(T))\ge0$, 
and as this holds for all $\alpha\in\DD$,
we conclude that $w(f(T))\le1$. 
This shows that $(\overline{\Omega},z_0)$ is a $W$-spectral pair
and establishes property (i).

Lastly we prove that (i)$\Rightarrow$(ii).
Assume  (i) holds. 
Let $h$ be a conformal mapping of $\Omega$ onto the right half-plane 
such that $h(z_0)=1$.
Define $g:=(1-h)/(1+h)$. 
Then $g$ is a conformal mapping of $\Omega$ onto $\DD$ such that $g(z_0)=0$.
As $\Omega$ is a Jordan domain, 
$g$ extends to a homeomorphism of $\overline{\Omega}$ onto $\overline{\DD}$.
By Mergelyan's theorem, 
there exist polynomials $(f_n)$ such that $f_n\to g$ uniformly on $\overline{\Omega}$.
Since $g(z_0)=0$ and $|g|\le 1$ on $\overline{\DD}$, 
by adjusting the $f_n$ we can arrange that
$f_n(z_0)=0$ and $|f_n|\le1$ on $\overline{\DD}$ for all $n$. 
By the hypothesis (i),  
$(\overline{\Omega},z_0)$ is a $W$-spectral pair for $T$, 
so $w(f_n(T))\le 1$ for all $n$.
Letting $n\to\infty$, we obtain $w(g(T))\le 1$. Hence
\[
0\le \Re (I-g(T))=\Re (I- (I-h(T))(I+h(T))^{-1})=2\Re((I+h(T))^{-1}),
\]
whence $\Re (I+h(T))\ge0$. Therefore property (ii) holds.
\end{proof}

\section{Proof of Theorem~\ref{T:ellipse}}\label{S:ellipse}

The proof of Theorem~\ref{T:ellipse} is based on the following example
due to Crouzeix \cite{Cr16}. 

\begin{lemma}\label{L:ellipse}
Let $a>b>0$, let $c:=\sqrt{a^2-b^2}$, 
and let $T$ be the $2\times 2$ matrix
\[ 
T:=\begin{pmatrix}
c & 2b\\
0 &-c
\end{pmatrix}.
\]
Then $\sigma(T)=\{-c,c\}$ and $W(T)=\overline{\Omega}$, where
\[
\Omega:=\Bigl\{x+iy: \frac{x^2}{a^2}+\frac{y^2}{b^2}<1\Bigr\}.
\]
Further, if $f:\Omega\to\DD$ is a conformal mapping such that $f(0)=0$,
then $w(f(T))>1$.
\end{lemma}

The identification of $\sigma(T)$ and $W(T)$ is standard  
(see  \cite[Example~3, pp 2--3]{GR97}).
The point of the lemma is the inequality $w(f(T))>1$.
Crouzeix's proof of this is quite complicated. 
It depends upon an explicit representation of $f$ 
as an infinite series involving Chebyshev polynomials. 
We give a simpler proof based on Schwarz's lemma.

\begin{proof}[Proof of Lemma~\ref{L:ellipse}]
Let us begin by observing that $f$ is an odd function.
Indeed, set $\phi:=m\circ f\circ m\circ f^{-1}$, where $m(z)=-z$.
Then $\phi$ is a conformal self-map of $\DD$ such that $\phi(0)=0$ and $\phi'(0)=1$, 
so by Schwarz's lemma $\phi(z)\equiv z$. Thus $f(-z)=-f(z)$.
 
As $f$ is odd, it can be written as $f(z)=zg(z^2)$, 
where $g$ is holomorphic on the set $\{z^2:z\in \Omega\}$. 
Since $T^2=c^2I$, it follows  that
\[
f(T)=Tg(T^2)=Tg(c^2I)=g(c^2)T=(f(c)/c)T.
\]
In particular, the numerical radius of $f(T)$ is given by 
\[
w(f(T))=|f(c)/c|w(T)=|f(c)|a/c.
\] 
We shall now prove that $|f(c)|a/c>1$.

Let $\psi:=h \circ f^{-1}$, where $h(z):=z/a$.
Then  $\psi$ maps $\DD$ to a proper subset of itself and fixes~$0$.
By Schwarz's lemma, 
we have $|\psi(z)|<|z|$ for all $z\in\DD\setminus\{0\}$. 
In particular, taking $z:=f(c)$, we obtain
$c/a<|f(c)|$, as claimed.
\end{proof}

\begin{proof}[Proof of Theorem~\ref{T:ellipse}]
We keep the notation of Lemma~\ref{L:ellipse}.
Since $\Omega$ is a Jordan domain with analytic boundary, 
$f$ extends to be holomorphic on a neighborhood of~$\overline{\Omega}$. 
By Runge's theorem there exist polynomials 
$f_n$ converging uniformly to $f$ on $\overline{\Omega}$. 
By adjusting the $f_n$, 
we can further arrange that $f_n(0)=0$ and $|f_n|\le1$ on $\overline{\Omega}$. 
Since $w(f(T))>1$, 
we have $w(f_n(T))>1$ for all large enough $n$. 
This shows that $(\overline{\Omega},0)$ is not a $W$-spectral pair for $T$, 
and completes the proof of the theorem.
\end{proof}

\begin{remarks}
(i) Let $T$ be any non-self-adjoint operator satisfying $T^2=I$. Then
$P:=(T+I)/2$ is a non-self-adjoint idempotent so, by \cite[Theorem~2.3]{SS10},
$W(P)$ is an ellipse with foci at $0,1$ and eccentricity $1/\|P\|<1$.
It follows that $W(T)$ is an ellipse of eccentricity $<1$ centred at $0$.
The proof of Theorem~\ref{T:ellipse} now shows that $(W(T),0)$ is not a $W$-spectral pair for $T$.

(ii) As remarked by one of the referees 
(in response to a question posed in an earlier version of the paper),
numerical experiments show that the phenomenon exhibited in Theorem~\ref{T:ellipse} 
also occurs if we replace the ellipse by a rectangle, or even by a square.
Specifically, writing $f$ for the conformal map of $(-1,1)^2$ onto $\DD$ such that $f(0)=0$,
there exists a $3\times 3$ matrix $T$ such that 
$W(T)\subset (-1,1)^2$ but $f(W(T))\not\subset\overline{\DD}$.
\end{remarks}

\begin{acknowledgement}
We are grateful to both the anonymous referees for their careful reading of the paper and their insightful comments.\end{acknowledgement}


\begin{thebibliography}{99}

\bibitem{AHR08}
J. Agler, J. Harland and B. J. Raphael,
{\em Classical function theory, operator dilation theory, and machine-computation on multiply connected domains},
Mem. Amer. Math. Soc. {\bf 191} (2008).

\bibitem{Be63}
C. A. Berger,
{\em Normal dilations},
Thesis, Cornell University, 1963.

\bibitem{Be65}
C. A. Berger, 
{\em A strange dilation theorem}, 
Abstract 625--152,
Notices Amer. Math. Soc. {\bf 12} (1965), 590.

\bibitem{BS67}
C. A. Berger, J. G. Stampfli,  
{\em Mapping theorems for the numerical range},
Amer. J. Math. {\bf 89} (1967), 1047--1055. 

\bibitem{Cr16}
M. Crouzeix, private communication.

\bibitem{Dr08}
S. W. Drury,
{\em Symbolic calculus of operators with unit numerical radius},
Lin. Alg. Appl. {\bf 428} (2008), 2061--2069.

\bibitem{Fo59}
C. Foia\c s,
{\em Certaines applications des ensembles spectraux. I. Mesure harmonique-spectrale}
Acad. R. P. Rom\^\i ne. Stud. Cerc. Mat. {\bf 10} (1959), 365--401. 

\bibitem{GR97}
K. E. Gustafson, D. K. M. Rao, 
\emph{Numerical Range. The Field of Values of Linear Operators and Matrices},
Springer-Verlag, New York, 1997.

\bibitem{Ha57}
P. R. Halmos, 
{\em Introduction to Hilbert Space and the Theory of Spectral Multiplicity},
reprint of the second (1957) edition, AMS Chelsea Publishing, Providence, RI, 1998. 

\bibitem{Ka65}
T. Kato,
{\em Some mapping theorems for the numerical range},
Proc. Japan Acad. {\bf 41} (1965), 652--655. 

\bibitem{KMR16}
H. Klaja, J. Mashreghi, T. Ransford,
{\em On mapping theorems for numerical range},
Proc. Amer. Math. Soc. {\bf 144} (2016), 3009--3018.

\bibitem{Le63}
A. Lebow, 
{\em On von Neumann's theory of spectral sets},
J. Math. Anal. Appl. {\bf 7} (1963), 64--90.

\bibitem{Pa02}
V. Paulsen,
{\em Completely Bounded Maps and Operator Algebras},
Cambridge University Press, Cambridge, 2002.

\bibitem{SS10}
V. Simoncini, D. B. Szyld, 
{\em On the field of values of oblique projections},
Lin. Alg. Appl. {\bf 433} (2010),  810--818. 

\bibitem{SN53}
B. Sz.-Nagy,
{\em Sur les contractions de l'espace de Hilbert},
 Acta Sci. Math. Szeged {\bf 15} (1953), 87--92.
 
\bibitem{vN51}
J. von Neumann,
{\em Eine Spektraltheorie f\"ur allgemeine Operatoren eines unit\"aren Raumes},
Math. Nachr. {\bf 4} (1951), 258--281. 
        
\end{thebibliography}
\end{document}